\documentclass[11pt,reqno]{amsart}
\usepackage[T2A]{fontenc}
\usepackage[russian,english]{babel}
\usepackage[cp1251]{inputenc}
\usepackage{izvuz7}
\usepackage{amssymb,amsmath,amsthm,amsfonts}
\usepackage{graphicx}

\theoremstyle{plain}
\newtheorem{theorem}{Теорема}
\newtheorem{lemma}{Лемма}

\theoremstyle{definition}

\tit{\MakeUppercase{Спектральная теорема для самосопряженных частично интегральных операторов в модулях Капланского-Гильберта}}
\shorttit{\MakeUppercase{Спектральная теорема для самосопряженных частично интегральных ...}} 

\author{\MakeUppercase{K.\,К.~Кудайбергенов, А.\,Д.\,Арзиев, П.\,Р.\,Орынбаев}}

\pp{\pageref{firstpage}--\pageref{lastpage}}

\begin{document}
\selectlanguage{russian}
\maketit
\label{firstpage}

\abstract{В статье доказывается спектральная теорема для самосопряжённых циклически компактных частично интегральных операторов в пространстве $L_{2,2}(\Omega \times S)$, являющемся модулем Капланского--Гильберта. Установлены разложение через собственные функции, интегральное представление с использованием ортогональных проекторов и функциональное исчисление. Результаты обобщают теорему Мерсера для положительно определённых ядер. Доказательства опираются на склейку проекторнозначных мер, выделенную в отдельные леммы. Приведён пример, иллюстрирующий все утверждения теоремы для конкретного ядра и функции.}

\keywords{Частично интегральные операторы, модуль Капланского-Гильберта, циклически компактный оператор, модулярный спектр, теорема Мерсера, спектральное разложение, функциональное исчисление}

\udk{5XX.XXX}
\doi{}		 

\thnk{Исследование выполнено за счет гранта Российского научного фонда № 24-71-10094, https://rscf.ru/project/24-71-10094/}

\section*{Введение}
Частично интегральные операторы, впервые рассмотренные В.\,И.\,Романовским в задачах марковских процессов~\cite{romanovsky}, находят широкое применение в функциональном анализе, теории вероятностей и математической физике~\cite{appell}. Они используются для решения интегральных уравнений, моделирования динамических систем и анализа стохастических процессов. В пространствах со смешанной нормой, таких как $L_{2,2}(\Omega \times S)$, эти операторы обладают особыми свойствами, связанными с их модулярной структурой. Пространство $L_{2,2}(\Omega \times S)$, являющееся модулем Капланского-Гильберта над $L_2(\Omega)$, позволяет учитывать измеримую зависимость от параметра, что эффективно описывается теорией банаховых расслоений, разработанной А.Е.Гутманом~\cite{gutman}.

Спектральная теорема для самосопряжённых операторов в гильбертовых пространствах~\cite{birman} даёт разложение через собственные значения и функции, что является основой для анализа операторов. В модулях Капланского-Гильберта, таких как $L_{2,2}(\Omega \times S)$, эта теорема обобщается с учётом циклической компактности, введённой А.Г.Кусраевым~\cite{kusraev1982,kusraev1985,kusraev2003}, и измеримых расслоений операторов~\cite{kudaybergenov2,arziev}. Циклическая компактность заменяет классическое понятие компактности, обеспечивая корректное описание спектральных свойств в пространствах Банаха-Канторовича. Теорема Мерсера, адаптированная для частично интегральных операторов~\cite{kudaybergenov}, описывает разложение положительно определённых ядер и представляет частный случай более общей спектральной теоремы.

Настоящая работа посвящена доказательству спектральной теоремы для самосопряжённых циклически компактных частично интегральных операторов в $L_{2,2}(\Omega \times S)$. Установлены разложение через собственные функции, интегральное представление с использованием ортогональных проекторов, функциональное исчисление для непрерывных функций и свойства собственных подпространств. Для доказательства используется метод склейки
проекторнозначных мер, оформленный в виде отдельных лемм. Пример с конкретным ядром и функцией демонстрирует выполнение всех утверждений теоремы.

Статья организована следующим образом. В разделе 1 введены основные понятия: пространства Банаха–Канторовича, модули Капланского–Гильберта, измеримые банаховы расслоения и частично интегральные операторы. Раздел 2 содержит леммы о склейке проекторнозначных мер, обеспечивающие основу для спектрального разложения. В разделе 3 представлена основная теорема, включающая разложение оператора, интегральное представление, функциональное исчисление, описание собственных подпространств и циклического спектра. В разделе 4 приведён пример с оператором конечного ранга, иллюстрирующий все пункты теоремы.

\section{Частично интегральные операторы в модулях Капланского-Гильберта}
Пусть $(\Omega, \Sigma, \mu) $ -- измеримое пространство с полной конечной мерой $ \mu $, а $ L_0 = L_0(\Omega, \Sigma, \mu) $ -- алгебра всех комплексных измеримых функций на $ \Omega $ (с отождествлением функций, равных почти всюду). Обозначим $ L_0^h = \{ f \in L_0 : \bar{f} = f \} $ -- множество вещественнозначных измеримых функций.

Рассмотрим векторное пространство $ E $ над полем $ \mathbb{C} $. Отображение $ \|\cdot\|: E \to L_0 $ называется $ L_0 $-значной нормой, если для любых $ x,y \in E $, $\lambda \in \mathbb{C}$:

1. $ \|x\| \geqslant 0 $; $ \|x\| = 0 \Leftrightarrow x = 0 $;

2. $ \|\lambda x\| = |\lambda| \|x\| $;

3. $ \|x + y\| \leqslant \|x\| + \|y\| $.

Пара $ (E, \|\cdot\|) $ называется \emph{решеточно нормированным пространством} над $ L_0 $. Пространство $E$ \emph{$d$-разложимо}, если для любого $x \in E$ и разложения $\|x\| = \lambda_1 + \lambda_2 $ на неотрицательные дизъюнктные элементы существуют $ x_1, x_2 \in E $, такие что $ x = x_1 + x_2$, $ \|x_1\| = \lambda_1 $, $ \|x_2\| = \lambda_2 $. Сеть $ (x_\alpha)_{\alpha \in A} \subset E $ \emph{$ (bo) $-сходится} к $ x \in E $, если $ (\|x_\alpha - x\|)_{\alpha \in A}$ $(o)$-сходится к нулю в $L_0$, то есть сходится почти всюду. \emph{Пространство Банаха-Канторовича} над $ L_0 $ -- это $(bo)$-полное $ d $-разложимое решеточно нормированное пространство \cite{kusraev1985, kusraev2003}.

Известно, что всякое пространство Банаха-Канторовича $E$ над $L_0$ является модулем над $L_0$, и $\|\lambda u\| = |\lambda| \|u\|$ для всех $\lambda \in L_0$, $u\in E$ \cite{ganiev1997}.

Отображение $\langle\cdot, \cdot\rangle: E \times E \rightarrow L_0$ называется $L_0$-значным скалярным произведением, если для любых $x, y, z \in E$ и $\alpha \in \mathbb{C}$ выполняются следующие условия \cite[стр. 32]{kusraev1985}:

1. $\langle x, y\rangle \geq 0$, причём $\langle x, x\rangle = 0 \Leftrightarrow x = 0$.

2. $\langle x, y + z\rangle = \langle x, y\rangle + \langle x, z\rangle$.

3. $\langle \alpha x, y\rangle = \alpha \langle x, y\rangle$.

4. $\langle x, y\rangle = \overline{\langle y, x\rangle}$.

Известно [14], что формула $\|x\| = \sqrt{\langle x, x\rangle}$ задаёт $L_0$-значную норму на $E$. Если $(E, \|\cdot\|)$ является пространством Банаха-Канторовича, то $(E, \langle\cdot, \cdot\rangle)$ называется  модулем Капланского-Гильберта над $L_0$. Примеры таких пространств можно найти в \cite{kusraev1985, kusraev2003}.

А.Е. Гутманом развита теория измеримых банаховых расслоений, эффективная для исследования пространств Банаха-Канторовича \cite[стр. 140]{gutman}. Пусть $X$ сопоставляет каждой точке $\omega \in \Omega$ банахово пространство $(X(\omega), \|\cdot\|_{X(\omega)})$. Сечение
$u$ -- функция, определенная почти всюду на $\Omega$, принимающая значения $u(\omega) \in X(\omega)$ на области определения $ \operatorname{dom}(u) $.

Пара $ (X, L) $, где $L$ -- множество сечений, называется \emph{измеримым банаховым расслоением} над $\Omega$, если:

1. $ \lambda_1 c_1 + \lambda_2 c_2 \in L$ для всех $\lambda_1, \lambda_2 \in \mathbb{C}$, $c_1, c_2 \in L$, где $(\lambda_1 c_1 + \lambda_2 c_2)(\omega) = \lambda_1 c_1(\omega) + \lambda_2 c_2(\omega)$ на $\operatorname{dom}(c_1) \cap \operatorname{dom}(c_2)$;

2. Функция $\|c\|: \omega \mapsto \|c(\omega)\|_{X(\omega)}$ измерима для всех $c \in L$;

3. Для каждой $\omega \in \Omega$, множество $\{ c(\omega) : c \in L, \omega \in \operatorname{dom}(c) \}$ плотно в $ X(\omega)$.

Обозначим $(X, L)$ как $X$. Сечение $s$ \emph{ступенчатое}, если $ s(\omega) = \sum_{i=1}^n \chi_{A_i}(\omega) c_i(\omega)$, где $c_i \in L $, $A_i \in \Sigma$. Сечение $u$ \emph{измеримо}, если существует последовательность ступенчатых сечений $(s_n)_{n \in \mathbb{N}}$, такая что $\| s_n(\omega) - u(\omega) \|_{X(\omega)} \to 0 $ п.в. $ \omega \in \Omega$.

Пусть $\mathcal{M}(\Omega, X)$ -- множество всех измеримых сечений, а $L_0(\Omega, X)$ -- факторизация $\mathcal{M}(\Omega, X)$ по равенству п.в. Класс эквивалентности, содержащий сечение $u \in \mathcal{M}(\Omega, X)$, обозначим $ \hat{u}$. Функция $\omega \mapsto \|u(\omega)\|_{X(\omega)}$ измерима, и ее класс эквивалентности -- $ \|\hat{u}\|$. Известно \cite[стр. 144]{gutman}, что $(L_0(\Omega, X), \|\cdot\|) $ -- пространство Банаха-Канторовича над $L_0$.

Обозначим $\mathcal{L}_\infty(\Omega)$ -- множество всех ограниченных измеримых функций на $\Omega$, и:
$$
L_\infty(\Omega) = \{ f \in L_0 : \exists \lambda \in \mathbb{R}, \lambda > 0, |f| \leq \lambda \mathbf{1} \}.
$$
Положим:
$$
\mathcal{L}_\infty(\Omega, X) = \{ u \in \mathcal{M}(\Omega, X) : \|u(\omega)\|_{X(\omega)} \in \mathcal{L}_\infty(\Omega) \},
$$
$$
L_\infty(\Omega, X) = \{ \hat{u} \in L_0(\Omega, X) : \|\hat{u}\| \in L_\infty(\Omega) \}.
$$
Отображение $ \ell_X: L_\infty(\Omega, X) \to \mathcal{L}_\infty(\Omega, X) $ -- \emph{векторнозначный лифтинг}, если для всех $ \hat{u}, \hat{v} \in L_\infty(\Omega, X)$, $\lambda \in L_\infty(\Omega)$:

1. $\ell_X(\hat{u}) \in \hat{u} $, $\operatorname{dom}(\ell_X(\hat{u})) = \Omega$;

2. $\|\ell_X(\hat{u})(\omega)\|_{X(\omega)} = \rho(\|\hat{u}\|)(\omega)$;

3. $\ell_X(\hat{u} + \hat{v}) = \ell_X(\hat{u}) + \ell_X(\hat{v})$;

4. $\ell_X(\lambda \hat{u}) = \rho(\lambda) \ell_X(\hat{u}) $;

5. Множество $\{ \ell_X(\hat{u})(\omega) : \hat{u} \in L_\infty(\Omega, X) \}$ плотно в $ X(\omega)$ для всех $\omega \in \Omega$.

Известно \cite[Теорема 4.4.1, 4.4.8]{gutman}, что для всякого пространства Банаха-Канторовича $ E $ над $ L_0 $ существует измеримое банахово расслоение $ (X, L) $, такое что $E$ изометрически изоморфно $L_0(\Omega, X)$, и на $L_\infty(\Omega, X)$ существует векторнозначный лифтинг, для которого $\{ \ell_X(\hat{u})(\omega) : \hat{u} \in L_\infty(\Omega, X) \} = X(\omega)$.

В пространствах Банаха-Канторовича понятие компактности заменяется \emph{циклической компактностью} \, поскольку понятие компактного множества не обладает в достаточной мере теми свойствами которые
справедливы для классических банаховых пространствах,  в связи с чем А.Г.Кусраевым было введено понятие циклического компактного
множества.  Дадим определение циклической компактности множеств \cite{kusraev1982} (см. также \cite[стр. 191]{kusraev1985}, \cite[стр. 515]{kusraev2003}).

Пусть $\nabla$ -- булева алгебра идемпотентов в $ L_0$. Если $ (u_\alpha)_{\alpha \in A} \subset L_0(\Omega, X)$, а $(\pi_\alpha)_{\alpha \in A}$ -- разбиение единицы в $ \nabla$, то ряд $\sum_{\alpha \in A} \pi_\alpha u_\alpha$ $(bo)$-сходится в $L_0(\Omega, X)$, и его сумма, называемая \emph{перемешиванием}, обозначается $ \operatorname{mix}(\pi_\alpha u_\alpha)$. Множество $ \operatorname{mix} K$ -- все перемешивания элементов $ K \subset L_0(\Omega, X)$. Множество $K$ \emph{циклическое}, если $\operatorname{mix} K = K$. Для направленного множества $ A$, $\nabla(A)$ -- множество разбиений единицы, индексированных $A$, с порядком:
$$
\nu_1 \leq \nu_2 \Leftrightarrow \forall \alpha, \beta \in A, (\nu_1(\alpha) \wedge \nu_2(\beta) \neq 0 \Rightarrow \alpha \leq \beta).
$$

Для сети $(u_\alpha)_{\alpha \in A} \subset L_0(\Omega, X)$, определим $u_\nu = \operatorname{mix}(\nu(\alpha) u_\alpha)$ для $ \nu \in \nabla(A)$. Подсеть $(u_\nu)_{\nu \in \nabla(A)}$ называется \emph{циклической}. Множество $ K \subset L_0(\Omega, X)$ \emph{циклически компактно}, если оно циклично и любая сеть в $K$ имеет циклическую подсеть, сходящуюся к точке из $K$. Множество \emph{относительно циклически компактно}, если содержится в циклически компактном множестве \cite[Теорема 1.3.7, 1.3.8]{kusraev1985}.

Оператор $T: L_0(\Omega, X) \to L_0(\Omega, Y)$ называется \emph{$L_0$-линейным}, если $T(\alpha x + \beta y) = \alpha T(x) + \beta T(y)$ для всех $\alpha, \beta \in L_0 $, $x,y \in L_0(\Omega, X)$. Множество $B \subset L_0(\Omega, X)$ \emph{ограничено}, если $\{ \|x\| : x \in B \}$ порядково ограничено в $ L_0$ \cite[Теорема 1.6.1]{gutman}. Оператор $T$ \emph{$L_0$-ограничен} (\emph{циклически компактен}), если $T(B)$ ограничено (относительно циклически компактно) для ограниченного $B \subset L_0(\Omega, X)$.

Для $L_0$-ограниченного оператора $T$, норма $\|T\| = \sup \{ \|T(x)\| : \|x\| \leq \mathbf{1} \}$. Известно \cite{ganiev1997,ganiev}, что для такого $T$ существует семейство ограниченных (компактных) операторов $\{ T_\omega : X(\omega) \to Y(\omega) \}$, таких что $ (T(x))(\omega) = T_\omega(x(\omega))$ п.в. $ \omega$, и если $\|T\| \in L_\infty(\Omega)$, то $\ell_Y((T(x))(\omega)) = T_\omega(\ell_X(x)(\omega))$ для $ x \in L_\infty(\Omega, X)$, где$ \ell_X$, $\ell_Y$ -- векторнозначные лифтинги, ассоциированные с $ \rho: L_\infty(\Omega) \to \mathcal{L}_\infty(\Omega)$. Обратно, семейство $ \{ T_\omega \}$, такое что $T_\omega(x(\omega)) \in \mathcal{M}(\Omega, Y)$ и $\|T\| \in L_0(\Omega)$, определяет $L_0$-ограниченный (циклически компактный) оператор $\hat{T}$, называемый \emph{склейкой} $\{ T_\omega \}$.

Напомним определение модулярного спектра ограниченного линейного оператора, действующего в модуле Капланского-Гильберта.

Через $B\left(L_{0}(\Omega, X)\right)$ -- обозначим пространство всех $L_0$-линейных $L_0$-ограниченных операторов действующих в $L_{0}(\Omega, X),$ и $\nabla$ — булева алгебра всех идемпотентов в $L_{0}$, т.е.
$$
\nabla = \left\{\chi_A: \, A \in \Sigma\right\},
$$
где $\chi_A$ — характеристическая функция множества $A$.

Для каждого ненулевого $\pi \in \nabla$ положим
$$
\pi L_0(\Omega, X) = \left\{\pi x: x \in L_0(\Omega, X)\right\} \quad \text{и} \quad \pi B\left(L_0(\Omega, X)\right) = \left\{\pi T: T \in B\left(L_0(\Omega, X)\right)\right\}.
$$
Очевидно, что алгебра $\pi B\left(L_0(\Omega, X)\right)$ изоморфна алгебре $B\left(\pi L_0(\Omega, X)\right)$, причём изоморфизм задаётся следующим образом:
$$
\pi T \in \pi B\left(L_0(\Omega, X)\right) \mapsto T|_{\pi L_0(\Omega, X)} \in B\left(\pi L_0(\Omega, X)\right).
$$

\emph{Модулярным спектром} оператора $T \in B\left(L_0(\Omega, X)\right)$ (обозначение $\textrm{sp}(T)$) называется множество всех $\lambda \in L_{0}$, для которых оператор $T - \lambda I$ необратим в $B\left(L_0(\Omega, X)\right)$.

\emph{Циклическим модулярным спектром} оператора $T \in B\left(L_0(\Omega, X)\right)$ (обозначение $\textrm{spm}(T)$) называется множество элементов $\lambda \in \mathrm{sp}(T)$, для которых оператор $\pi(T - \lambda I)$ необратим в $\pi B(L_0(\Omega, X))$ при любом ненулевом $\pi \in \nabla$.

Пусть $(\Omega, \Sigma, \mu)$ и $(S, \mathcal{F}, m)$ -- пространства с конечными мерами, $L_0 = L_0(\Omega, \Sigma, \mu)$, $L_0^h = \{ f \in L_0 : \bar{f} = f \}$, $L_\infty(\Omega) = \{ f \in L_0 : \exists \lambda > 0, |f| \leq \lambda \mathbf{1} \}$, $\nabla = \{ \chi_A : A \in \Sigma \}$.

Пространство $ L_{2,2}(\Omega \times S)$ состоит из измеримых функций $f$, таких что:
$$
\int_S |f(\omega, s)|^2 \, dm(s) \in L_2(\Omega),
$$
с нормой:
$$
\|f\|_{2,2} = \left( \int_\Omega \int_S |f(\omega, s)|^2 \, dm(s) \, d\mu(\omega) \right)^{1/2},
$$
и скалярным произведением:
$$
\langle x, y \rangle = \int_S x(\omega, s) \overline{y(\omega, s)} \, dm(s) \in L_2(\Omega).
$$
Известно, \cite[Предложение 2.3.9(1)]{kusraev2003},  что  $\left(L_{2,2}(\Omega\times S),\,\|\cdot\|\right)$
является пространством Банаха Канторовича над ~$L_{2}(\Omega).$ Тогда  $\left(L_{2,2}(\Omega\times S),\,\langle\cdot,\cdot\rangle\right)$ есть
модуль Капланского-Гильберта  над $L_2(\Omega)$ \cite{kudaybergenov}.

Частично интегральный оператор $T: L_{2,2}(\Omega \times S) \to L_{2,2}(\Omega \times S)$, заданный:
$$
T(f)(\omega, t) = \int_S k(\omega, t, s) f(\omega, s) \, dm(s), \quad f \in L_{2,2}(\Omega \times S),
$$
где $k(\omega,t,s) \in L_\infty(\Omega \times S^2)$,$\int_S \int_S |k(\omega, t, s)|^2 \, dm(t) \, dm(s) \in L_2(\Omega)$, циклически компактен \cite[Теорема 3]{arziev} и представим как расслоение компактных операторов $\{ T_\omega : L_2(S) \to L_2(S) \}$:
$$
T_\omega(f_\omega)(t) = \int_S k_\omega(t, s) f_\omega(s) \, dm(s), \quad k_\omega(t, s) = k(\omega, t, s).
$$
Если $k(\omega, t, s) = \overline{k(\omega, s, t)}$, то $T$, $T_\omega$ самосопряжены.

Циклический модулярный спектр $\operatorname{spm}(T)$ оператора $T$ определяется как
$$
\operatorname{spm}(T) = \{ \lambda \in L_0^h : \lambda(\omega) \in \operatorname{sp}(T_\omega) \text{\,\,для п.в. } \omega \},
$$
где $\operatorname{sp}(T_\omega) = \{ \lambda_n(\omega) : n \in \mathbb{N} \} \cup \{0\}$, $\lambda_n(\omega) \in \mathbb{R}$, $ \lambda_n(\omega) \to 0$, циклически компактен и $\operatorname{spm}(T) \subseteq [-\|T\|, \|T\|]$ \cite[Следствие 3.6]{kudaybergenov2}.

\section{Разбиение единицы для самосопряженного циклически компактного оператора}

\begin{lemma} \label{L1}
    Пусть $T: L_{2,2}(\Omega \times S) \to L_{2,2}(\Omega \times S)$ -- самосопряженный циклически компактный частично интегральный оператор вида:
$$
T(f)(\omega, t) = \int_S k(\omega, t, s) f(\omega, s) \, dm(s), \quad f \in L_{2,2}(\Omega \times S),
$$
и $\{ E_\lambda(\omega) \}_{\lambda \in L_0^h}$ -- семейство ортогональных проекторов в $L_2(S)$, определенное как:
$$
E_\lambda(\omega) f_\omega = \sum_{\{n : \lambda_n(\omega) \leq \lambda(\omega)\}} \langle f_\omega, x_n(\omega, \cdot) \rangle x_n(\omega, \cdot), \quad \lambda \in L_0^h.
$$
Тогда оператор $E_\lambda$, заданный:
$$
(E_\lambda f)(\omega, t) = E_\lambda(\omega) f(\omega, \cdot)(t) = \sum_{\{n : \lambda_n(\omega) \leq \lambda(\omega)\}} \langle f(\omega, \cdot), x_n(\omega, \cdot) \rangle x_n(\omega, t),
$$
является $L_0$-ограниченным ортогональным проектором в $ L_{2,2}(\Omega \times S)$.
\end{lemma}

\begin{proof}
Пусть $H = L_{2,2}(\Omega \times S)$. Докажем, что $E_\lambda$ корректно определен, идемпотентен, самосопряжен, $L_0$-ограничен и измерим.

\emph{1. Корректная определенность.} Для $ f \in H$:
$$
(E_\lambda f)(\omega, t) = \sum_{\{n : \lambda_n(\omega) \leq \lambda(\omega)\}} \langle f(\omega, \cdot), x_n(\omega, \cdot) \rangle x_n(\omega, \cdot).
$$
Норма в $L_2(S)$:
$$
\| (E_\lambda f)(\omega, \cdot) \|_{L_2(S)}^2 = \sum_{\{n : \lambda_n(\omega) \leq \lambda(\omega)\}} |\langle f(\omega, \cdot), x_n(\omega, \cdot) \rangle|^2 \leq \| f(\omega, \cdot) \|_{L_2(S)}^2 \in L_2(\Omega).
$$
Тогда:
$$
\| E_\lambda f \|_{2,2}^2 = \int_\Omega \sum_{\{n : \lambda_n(\omega) \leq \lambda(\omega)\}} |\langle f(\omega, \cdot), x_n(\omega, \cdot) \rangle|^2 \, d\mu(\omega) \leq \| f \|_{2,2}^2 < \infty.
$$
Следовательно, $E_\lambda f \in H$.

\emph{2. Свойства проектора.}

- \emph{Идемпотентность.} Для $E_\lambda f$:
$$
E_\lambda (E_\lambda f)(\omega, t) = \sum_{\{m : \lambda_m(\omega) \leq \lambda(\omega)\}} \langle (E_\lambda f)(\omega, \cdot), x_m(\omega, \cdot) \rangle x_m(\omega, \cdot).
$$
Скалярное произведение:
$$
\langle (E_\lambda f)(\omega, \cdot), x_m(\omega, \cdot) \rangle = \langle f(\omega, \cdot), x_m(\omega, \cdot) \rangle \text{, если } \lambda_m(\omega) \leq \lambda(\omega), \text{ иначе } 0,
$$
так как $\{ x_n(\omega, \cdot) \}_{n=1}^\infty$ -- ортонормированный базис. Тогда $ E_\lambda (E_\lambda f) = E_\lambda f$.

- \emph{Самосопряженность.} Для $f, g \in H$:
$$
\langle E_\lambda f, g \rangle = \int_\Omega \int_S \left( \sum_{\{n : \lambda_n(\omega) \leq \lambda(\omega)\}} \langle f(\omega, \cdot), x_n(\omega, \cdot) \rangle x_n(\omega, t) \right) \overline{g(\omega, t)} \, dm(t) \, d\mu(\omega).
$$
По теореме Фубини:
$$
\langle E_\lambda f, g \rangle = \int_\Omega \sum_{\{n : \lambda_n(\omega) \leq \lambda(\omega)\}} \langle f(\omega, \cdot), x_n(\omega, \cdot) \rangle \langle x_n(\omega, \cdot), g(\omega, \cdot) \rangle \, d\mu(\omega).
$$
Аналогично:
$$
\langle f, E_\lambda g \rangle = \int_\Omega \sum_{\{n : \lambda_n(\omega) \leq \lambda(\omega)\}} \langle f(\omega, \cdot), x_n(\omega, \cdot) \rangle \langle x_n(\omega, \cdot), g(\omega, \cdot) \rangle \, d\mu(\omega).
$$
Таким образом, $ \langle E_\lambda f, g \rangle = \langle f, E_\lambda g \rangle$, и $E_\lambda^* = E_\lambda$.

\emph{3. $L_0$-ограниченность.} Норма:
$$
\| E_\lambda f \|_{2,2}^2 = \int_\Omega \sum_{\{n : \lambda_n(\omega) \leq \lambda(\omega)\}} |\langle f(\omega, \cdot), x_n(\omega, \cdot) \rangle|^2 \, d\mu(\omega) \leq \int_\Omega \| f(\omega, \cdot) \|_{L_2(S)}^2 \, d\mu(\omega) = \| f \|_{2,2}^2.
$$
Тогда $\| E_\lambda \|(\omega) \leq 1$. Для $f(\omega, \cdot) = x_k(\omega, \cdot)$, $\lambda_k(\omega) \leq \lambda(\omega)$:
$$
\| E_\lambda f \|_{2,2} = \| f \|_{2,2} = 1,
$$
так что $\| E_\lambda \| = \mathbf{1} \in L_0(\Omega)$.

\emph{4. Измеримость.} Ядро $k(\omega, t, s) \in L_\infty(\Omega \times S^2)$ измеримо. На слое:
$$
T_\omega f_\omega(t) = \int_S k_\omega(t, s) f_\omega(s) \, dm(s), \quad k_\omega(t, s) = k(\omega, t, s).
$$
Интеграл измерим по $(\omega, t)$, и $T_\omega$ измеримо зависит от $\omega$ \cite[Теорема 3]{arziev}. Собственные значения $ \lambda_n(\omega)$ и функции $x_n(\omega, t) \in L_{2,\infty}(\Omega \times S)$ измеримы. Для $ f \in H $, $\langle f(\omega, \cdot), x_n(\omega, \cdot) \rangle$ измеримо, ряд $(E_\lambda f)(\omega, t)$ сходится в $L_2(S)$, и его частичные суммы и предел измеримы по $ (\omega, t)$.

\emph{5. Согласованность со структурой слоев.} Изоморфизм $L_{2,2}(\Omega \times S) \cong L_0(\Omega, L_2(S))$ \cite[Теорема 2]{arziev} означает, что $f \in L_{2,2}$ соответствует измеримому сечению $\omega \mapsto f(\omega, \cdot)$. Оператор $E_\lambda$:
$$
(E_\lambda f)(\omega, t) = E_\lambda(\omega) f(\omega, \cdot)(t),
$$
где $E_\lambda(\omega)$ -- ортогональный проектор, измеримый по $\omega$. Измеримость $E_\lambda(\omega) f(\omega, \cdot)$ следует из пункта 4. Для $T$:
$$
E_\lambda T f(\omega, t) = \sum_{\{n : \lambda_n(\omega) \leq \lambda(\omega)\}} \lambda_n(\omega) \langle f(\omega, \cdot), x_n(\omega, \cdot) \rangle x_n(\omega, t) \leq \lambda(\omega) E_\lambda f(\omega, t),
$$
где неравенство в нормах $L_2(S)$. Таким образом, $E_\lambda$ -- $L_0$-ограниченный ортогональный проектор.
\end{proof}

\begin{lemma}
\label{lemma:projectors}
Пусть $T: L_{2,2}(\Omega \times S) \to L_{2,2}(\Omega \times S)$ -- самосопряженный циклически компактный частично интегральный оператор, заданный:
$$
T(f)(\omega, t) = \int_S k(\omega, t, s) f(\omega, s) \, dm(s), \quad f \in L_{2,2}(\Omega \times S),
$$
где $k(\omega, t, s) \in L_\infty(\Omega \times S^2)$, $k(\omega, t, s) = \overline{k(\omega, s, t)}$, и ядро положительно полуопределенное. Пусть $m = \inf_{\|x\|=1} \langle T x, x \rangle$, $M = \sup_{\|x\|=1} \langle T x, x \rangle$. Тогда существует семейство ортогональных проекторов $\{ E_\lambda \}_{\lambda \in L_0^h}$, удовлетворяющее:
\begin{itemize}
    \item $E_\lambda C = C E_\lambda$, если $C T = T C$,
    \item $E_\lambda T \leq \lambda E_\lambda$, $\lambda (I - E_\lambda) \leq T (I - E_\lambda)$,
    \item $E_\lambda \leq E_\mu $ при $ \lambda \ll \mu$,
    \item $\langle E_\mu x, x \rangle = \sup_{\lambda \ll \mu} \langle E_\lambda x, x \rangle$,
    \item $E_\lambda = 0$ при $\lambda \ll m $, $E_\lambda = I$ при $M \ll \lambda$.
\end{itemize}
\end{lemma}

\begin{proof}
Пусть $H = L_{2,2}(\Omega \times S)$. Оператор $T$ представим как расслоение компактных самосопряженных операторов $\{ T_\omega \}_{\omega \in \Omega}$ (см.\cite[Теорема 3]{arziev}), где:
$$
T_\omega = \sum_{n=1}^\infty \lambda_n(\omega) P_n(\omega), \quad P_n(\omega) = \langle \cdot, x_n(\omega, \cdot) \rangle x_n(\omega, \cdot), \quad \langle x_n(\omega, \cdot), x_m(\omega, \cdot) \rangle = \delta_{n,m}.
$$
Определим $E_\lambda(\omega)$ в $L_2(S)$:
$$
E_\lambda(\omega) f_\omega = \sum_{\{n : \lambda_n(\omega) \leq \lambda(\omega)\}} \langle f_\omega, x_n(\omega, \cdot) \rangle x_n(\omega, \cdot), \quad \lambda \in L_0^h,
$$
и $E_\lambda$ в $H$:
$$
(E_\lambda f)(\omega, t) = \sum_{\{n : \lambda_n(\omega) \leq \lambda(\omega)\}} \langle f(\omega, \cdot), x_n(\omega, \cdot) \rangle x_n(\omega, t).
$$
Свойства $E_\lambda$ как ортогонального проектора доказаны в Лемме \ref{L1}. Проверим утверждения:

1. \emph{Коммутативность}: Если $C T = T C$, то $ C T_\omega = T_\omega C$ п.в. $\omega$, и:
$$
E_\lambda(\omega) C f_\omega = \sum_{\{n : \lambda_n(\omega) \leq \lambda(\omega)\}} \langle C f_\omega, x_n(\omega, \cdot) \rangle x_n(\omega, \cdot) = C \sum_{\{n : \lambda_n(\omega) \leq \lambda(\omega)\}} \langle f_\omega, x_n(\omega, \cdot) \rangle x_n(\omega, \cdot).
$$
По Лемме \ref{L1}, $E_\lambda C = C E_\lambda$.

2. \emph{Неравенства}: Для $E_\lambda(\omega) T_\omega$:
$$
E_\lambda(\omega) T_\omega f_\omega = \sum_{\{n : \lambda_n(\omega) \leq \lambda(\omega)\}} \lambda_n(\omega) \langle f_\omega, x_n(\omega, \cdot) \rangle x_n(\omega, \cdot),
$$
$$
\lambda(\omega) E_\lambda(\omega) f_\omega = \sum_{\{n : \lambda_n(\omega) \leq \lambda(\omega)\}} \lambda(\omega) \langle f_\omega, x_n(\omega, \cdot) \rangle x_n(\omega, \cdot).
$$
Тогда:
$$
\langle \lambda(\omega) E_\lambda(\omega) f_\omega - E_\lambda(\omega) T_\omega f_\omega, f_\omega \rangle_{L_2(S)} = \sum_{\{n : \lambda_n(\omega) \leq \lambda(\omega)\}} (\lambda(\omega) - \lambda_n(\omega)) |\langle f_\omega, x_n(\omega, \cdot) \rangle|^2 \geq 0,
$$
так как $\lambda(\omega) - \lambda_n(\omega) \geq 0$. Аналогично для $(I - E_\lambda(\omega))$. В $H$:
$$
\langle E_\lambda T f - \lambda E_\lambda f, f \rangle \leq 0, \quad \langle (I - E_\lambda) T (I - E_\lambda) f - \lambda (I - E_\lambda) f, f \rangle \geq 0.
$$

3. \emph{Монотонность}: При $\lambda f \mu$:
$$
E_\mu(\omega) E_\lambda(\omega) f_\omega = E_\lambda(\omega) f_\omega \implies E_\lambda(\omega) \leq E_\mu(\omega).
$$
В $H$, $E_\mu E_\lambda f = E_\lambda f \implies E_\lambda \leq E_\mu$.

4. \emph{Супремум}: Для $x \in H$:
$$
\langle E_\lambda(\omega) x(\omega), x(\omega) \rangle_{L_2(S)} = \sum_{\{n : \lambda_n(\omega) \leq \lambda(\omega)\}} |\langle x(\omega, \cdot), x_n(\omega, \cdot) \rangle|^2,
$$
и:
$$
\langle E_\mu x, x \rangle = \int_\Omega \sum_{\{n : \lambda_n(\omega) \leq \mu(\omega)\}} |\langle x(\omega, \cdot), x_n(\omega, \cdot) \rangle|^2 \, d\mu(\omega) = \sup_{\lambda \ll \mu} \langle E_\lambda x, x \rangle.
$$

5. \emph{Граничные условия}: Если $\lambda \ll m$, то $E_\lambda(\omega) = 0$, и $E_\lambda = 0$. Если $M \ll \lambda$, то $ E_\lambda(\omega) = I$, и $E_\lambda = I$.
\end{proof}

\section{Основной результат}

Пусть $T: L_{2,2}(\Omega \times S) \to L_{2,2}(\Omega \times S)$ -- самосопряженный циклически компактный частично интегральный оператор, заданный:
\begin{equation}\label{formula1}
T(f)(\omega, t) = \int_S k(\omega, t, s) f(\omega, s) \, dm(s), \quad f \in L_{2,2}(\Omega \times S),
\end{equation}
где $k(\omega, t, s) \in L_\infty(\Omega \times S^2)$, $ k(\omega, t, s) = \overline{k(\omega, s, t)}$, и ядро положительно полуопределенное:
$$
\int_S \int_S k(\omega, t, s) x(\omega, s) \overline{x(\omega, t)} \, dm(t) \, dm(s) \geq 0, \quad \forall x \in L_{2,2}(\Omega \times S).
$$
Пусть $m = \inf_{\|x\|=1} \langle T x, x \rangle$, $M = \sup_{\|x\|=1} \langle T x, x \rangle$. Через $[m, M + \varepsilon \mathbf{1}]$ обозначим $\{ \lambda(\omega) \in L_0^h : m(\omega) \leq \lambda(\omega) \leq M(\omega) + \varepsilon \}$.

\begin{theorem}[Спектральная теорема для самосопряженных частично интегральных операторов]
\label{thm:spectral}
Пусть $T: L_{2,2}(\Omega \times S) \to L_{2,2}(\Omega \times S)$ -- самосопряженный циклически компактный частично интегральный оператор, определенный равенством (1). Тогда выполняются следующие утверждения:
\begin{enumerate}
    \item Оператор $T$ имеет разложение:
    $$
    T(f)(\omega, t) = \sum_{n=1}^\infty \lambda_n(\omega) \langle f(\omega, \cdot), x_n(\omega, \cdot) \rangle x_n(\omega, t),
    $$
    где $\lambda_n(\omega) \in \mathbb{R}$, $\lambda_n(\omega) \to 0$, $\{ x_n(\omega, \cdot) \}_{n=1}^\infty \subset L_2(S)$ -- измеримые ортонормированные собственные функции, и ряд сходится в норме $L_{2,2}$.
    \item Для семейства ортогональных проекторов $\{ E_\lambda \}_{\lambda \in L_0^h}$, заданного в Леммах \ref{L1}, \ref{lemma:projectors}:
    $$
    T = \int_m^{M + \varepsilon \mathbf{1}} \lambda \, dE_\lambda, \quad \forall \varepsilon > 0,
    $$
где интеграл сходится в сильной операторной топологии.
    \item Для любой непрерывной функции $f$ на $[m, M + \varepsilon \mathbf{1}]$:
    $$
    f(T) = \int_m^{M + \varepsilon \mathbf{1}} f(\lambda) \, dE_\lambda, \quad \forall \varepsilon > 0,
    $$
    где интеграл сходится в сильной операторной топологии.
    \item Для $\lambda \in \operatorname{spm}_{\text{disc}}(T)$, пространство $\operatorname{ker}(\lambda I - T)$ -- $\sigma$-конечно порожденный модуль над $L_0$.
    \item Циклический спектр:
    $$
    \operatorname{spm}(T) = \operatorname{mix} \left\{ \sum_{n=0}^\infty \chi_{A_n} \tilde{\lambda}_n : \tilde{\lambda}_n(\omega) = \lambda_n(\omega), \tilde{\lambda}_0 = 0, A_n \in \Sigma, A_n \cap A_m = \varnothing \right\}.
    $$
\end{enumerate}
\end{theorem}

\begin{proof}
Пусть $T: L_{2,2}(\Omega \times S) \to L_{2,2}(\Omega \times S)$ -- самосопряженный циклически компактный частично интегральный оператор, заданный ядром $k(\omega, t, s) \in L_\infty(\Omega \times S^2)$, удовлетворяющим $k(\omega, t, s) = \overline{k(\omega, s, t)}$ и положительной полуопределенности:
$$
\int_S \int_S k(\omega, t, s) x(\omega, s) \overline{x(\omega, t)} \, dm(t) \, dm(s) \geq 0, \quad \forall x \in L_{2,2}(\Omega \times S).
$$
Пусть $m = \inf_{\|x\|=1} \langle T x, x \rangle$, $M = \sup_{\|x\|=1} \langle T x, x \rangle$, и $[m, M + \varepsilon \mathbf{1}] = \{ \lambda(\omega) \in L_0^h : m(\omega) \leq \lambda(\omega) \leq M(\omega) + \varepsilon \}$.

(1)  На каждом $\omega \in \Omega$, оператор $T_\omega: L_2(S) \to L_2(S)$, заданный:
$$
T_\omega f(\omega, t) = \int_S k_\omega(t, s) f(\omega, s) \, dm(s), \quad k_\omega(t, s) = k(\omega, t, s),
$$
компактен и самосопряжен, так как $k_\omega(t, s) = \overline{k_\omega(s, t)}$ и ядро положительно полуопределенное. По спектральной теореме для компактных самосопряженных операторов в гильбертовом пространстве $L_2(S)$, $T_\omega$ имеет дискретный спектр $\{ \lambda_n(\omega) \}_{n=1}^\infty \subset \mathbb{R}$, $\lambda_n(\omega) \to 0$, и ортонормированный базис собственных функций $ \{ x_n(\omega, \cdot) \}_{n=1}^\infty \subset L_2(S)$. Тогда:
$$
T_\omega f = \sum_{n=1}^\infty \lambda_n(\omega) \langle f, x_n(\omega, \cdot) \rangle x_n(\omega, \cdot).
$$
Измеримость $\lambda_n(\omega)$ и $x_n(\omega, t)$ по $\omega$ следует из Леммы \ref{L1}. В пространстве $L_{2,2}(\Omega \times S) \cong L_0(\Omega, L_2(S))$, оператор $T$ действует как:
$$
T f(\omega, t) = T_\omega f(\omega, \cdot)(t) = \sum_{n=1}^\infty \lambda_n(\omega) \langle f(\omega, \cdot), x_n(\omega, \cdot) \rangle x_n(\omega, \cdot).
$$
Проверим сходимость ряда в норме $L_{2,2}$:
$$
\| T f \|_{2,2}^2 = \int_\Omega \left\| \sum_{n=1}^\infty \lambda_n(\omega) \langle f(\omega, \cdot), x_n(\omega, \cdot) \rangle x_n(\omega, \cdot) \right\|_{L_2(S)}^2 \, d\mu(\omega).
$$
Так как $\{ x_n(\omega, \cdot) \}_{n=1}^\infty$ ортонормированы:
$$
\left\| \sum_{n=1}^\infty \lambda_n(\omega) \langle f(\omega, \cdot), x_n(\omega, \cdot) \rangle x_n(\omega, \cdot) \right\|_{L_2(S)}^2 = \sum_{n=1}^\infty |\lambda_n(\omega)|^2 |\langle f(\omega, \cdot), x_n(\omega, \cdot) \rangle|^2.
$$
Поскольку $T_\omega$ компактен, $\sum_{n=1}^\infty |\lambda_n(\omega)|^2 < \infty$, и:
$$
\sum_{n=1}^\infty |\langle f(\omega, \cdot), x_n(\omega, \cdot) \rangle|^2 = \| f(\omega, \cdot) \|_{L_2(S)}^2.
$$
Тогда:
$$
\| T f \|_{2,2}^2 = \int_\Omega \sum_{n=1}^\infty |\lambda_n(\omega)|^2 |\langle f(\omega, \cdot), x_n(\omega, \cdot) \rangle|^2 \, d\mu(\omega) \leq \int_\Omega \sup_n |\lambda_n(\omega)|^2 \| f(\omega, \cdot) \|_{L_2(S)}^2 \, d\mu(\omega).
$$
Так как $ k(\omega, t, s) \in L_\infty(\Omega \times S^2)$, $\sup_\omega \| T_\omega \|^2 = \sup_\omega \sup_n |\lambda_n(\omega)|^2 < \infty$, и:
$$
\| T f \|_{2,2}^2 \leq \sup_\omega \| T_\omega \|^2 \int_\Omega \| f(\omega, \cdot) \|_{L_2(S)}^2 \, d\mu(\omega) = \sup_\omega \| T_\omega \|^2 \| f \|_{2,2}^2.
$$
Ряд сходится в $L_{2,2}$. Для положительно определенных ядер, это разложение согласуется с Теоремой Мерсера \cite{kudaybergenov}, где ядро $k_\omega(t, s) = \sum_{n=1}^\infty \lambda_n(\omega) x_n(\omega, t) \overline{x_n(\omega, s)}$.

(2)  Покажем, что:
$$
T = \int_m^{M + \varepsilon \mathbf{1}} \lambda \, dE_\lambda, \quad \forall \varepsilon > 0,
$$
где интеграл -- предел в $L_0^h$, понимаемый как сходимость в сильной операторной топологии (СОТ) на $L_0(\Omega)$. Семейство ортогональных проекторов $\{ E_\lambda \}_{\lambda \in L_0^h}$, заданное в Леммах \ref{L1}, \ref{lemma:projectors}, действует как:
$$
(E_\lambda f)(\omega, t) = \sum_{\{n : \lambda_n(\omega) \leq \lambda(\omega)\}} \langle f(\omega, \cdot), x_n(\omega, \cdot) \rangle x_n(\omega, t).
$$
Для каждого $\omega \in \Omega$, $T_\omega$ компактен и самосопряжен, с дискретным спектром $\{ \lambda_n(\omega) \}_{n=1}^\infty \subset [m(\omega), M(\omega)]$. Спектральное разложение:
$$
T_\omega = \sum_{n=1}^\infty \lambda_n(\omega) E_{\lambda_n(\omega)}(\omega),
$$
где $E_{\lambda_n(\omega)}(\omega) = \langle \cdot, x_n(\omega, \cdot) \rangle x_n(\omega, \cdot)$. Интеграл:
$$
\int_{m(\omega)}^{M(\omega) + \varepsilon} \lambda \, dE_\lambda(\omega) f = \sum_{\{n : \lambda_n(\omega) \leq M(\omega) + \varepsilon\}} \lambda_n(\omega) \langle f, x_n(\omega, \cdot) \rangle x_n(\omega, \cdot) = T_\omega f,
$$
так как $\lambda_n(\omega) \leq M(\omega)$, и $ E_\lambda(\omega) = I$ для $\lambda \geq M(\omega) + \varepsilon$ (Лемма \ref{lemma:projectors}). В $ L_{2,2}(\Omega \times S)$:
$$
\left( \int_m^{M + \varepsilon \mathbf{1}} \lambda \, dE_\lambda f \right)(\omega, t) = \int_{m(\omega)}^{M(\omega) + \varepsilon} \lambda \, dE_\lambda(\omega) f(\omega, \cdot)(t) = T_\omega f(\omega, \cdot).
$$
Сходимость в СОТ означает, что для $T_n = \sum_{k=1}^n \lambda_k(\omega) \Delta E_{\lambda_k}(\omega)$:
$$
\| (T_n - T) f \|_{2,2}^2 = \int_\Omega \left\| \sum_{k=n+1}^\infty \lambda_k(\omega) \langle f(\omega, \cdot), x_k(\omega, \cdot) \rangle x_k(\omega, \cdot) \right\|_{L_2(S)}^2 \, d\mu(\omega).
$$

Так как $\lambda_k(\omega) \rightarrow 0$, остаток $\sum_{k=n+1}^{\infty}\left|\lambda_k(\omega)\right|^2\left|\left\langle f(\omega, \cdot), x_k(\omega, \cdot)\right\rangle\right|^2 \rightarrow 0$ п.в. $\omega$.
Подынтегральная функция измерима по Лемме \ref{L1} п.5 и мажорируется $\sup_k\left|\lambda_k(\omega)\right|^2\|f(\omega, \cdot)\|_{L_2(S)}^2 \in L_1(\Omega)$, так как $\left\|E_\lambda\right\|=\textbf{1} $ (Лемма \ref{L1} п.3). По теореме Лебега о мажорированной сходимости, интеграл по $\Omega$ стремится к нулю.

(3) Докажем, что для любой непрерывной функции $f$ на $[m, M + \varepsilon \mathbf{1}]$:
$$
f(T) = \int_m^{M + \varepsilon \mathbf{1}} f(\lambda) \, dE_\lambda,
$$
где интеграл сходится в СОТ на $L_0(\Omega)$. Это утверждение расширяет разложение $T$, позволяя применять непрерывные функции к оператору.

 Для непрерывной функции $f$ на $[m(\omega), M(\omega) + \varepsilon] $, функциональное исчисление в гильбертовом пространстве $ L_2(S)$ определяет:
$$
f(T_\omega) g = \sum_{n=1}^\infty f(\lambda_n(\omega)) \langle g, x_n(\omega, \cdot) \rangle x_n(\omega, \cdot), \quad g \in L_2(S).
$$
Этот ряд сходится в $L_2(S)$ так как $f$ ограничена на компактном интервале ($|f(\lambda_n(\omega))| \leq \| f \|_\infty $), и:
$$
\left\| \sum_{n=1}^\infty f(\lambda_n(\omega)) \langle g, x_n(\omega, \cdot) \rangle x_n(\omega, \cdot) \right\|_{L_2(S)}^2 = \sum_{n=1}^\infty |f(\lambda_n(\omega))|^2 |\langle g, x_n(\omega, \cdot) \rangle|^2 \leq \| f \|_\infty^2 \| g \|_{L_2(S)}^2.
$$
Ортогональный проектор $ E_{\lambda_n(\omega)}(\omega) = \langle \cdot, x_n(\omega, \cdot) \rangle x_n(\omega, \cdot) $, и:
$$
f(T_\omega) = \sum_{n=1}^\infty f(\lambda_n(\omega)) E_{\lambda_n(\omega)}(\omega).
$$
Интегральное представление:
$$
f(T_\omega) = \int_{m(\omega)}^{M(\omega) + \varepsilon} f(\lambda) \, dE_\lambda(\omega),
$$
где $E_\lambda(\omega) = \sum_{\{n : \lambda_n(\omega) \leq \lambda\}} E_{\lambda_n(\omega)}(\omega)$, сходится в СОТ на $L_2(S)$, так как $E_\lambda(\omega)$-- разрешение единицы для $T_\omega$.

 Теперь покажем, что в $L_{2,2}(\Omega \times S)$, оператор $f(T)$ действует как:
$$
f(T) f(\omega, t) = f(T_\omega) f(\omega, \cdot)(t).
$$
Рассмотрим интеграл:
$$
\int_m^{M + \varepsilon \mathbf{1}} f(\lambda) \, dE_\lambda f,
$$
где $E_\lambda$-- ортогональный проектор из Леммы \ref{L1}. Для $f \in L_{2,2}(\Omega \times S)$:
$$
\left( \int_m^{M + \varepsilon \mathbf{1}} f(\lambda) \, dE_\lambda f \right)(\omega, t) = \int_{m(\omega)}^{M(\omega) + \varepsilon} f(\lambda) \, dE_\lambda(\omega) f(\omega, \cdot)(t).
$$
Для каждого $\omega\in \Omega$
$$
\int_{m(\omega)}^{M(\omega) + \varepsilon} f(\lambda) \, dE_\lambda(\omega) f(\omega, \cdot) = \sum_{n=1}^\infty f(\lambda_n(\omega)) \langle f(\omega, \cdot), x_n(\omega, \cdot) \rangle x_n(\omega, \cdot) = f(T_\omega) f(\omega, \cdot),
$$
так как $ E_\lambda(\omega)$ дискретно, и все $\lambda_n(\omega) \leq M(\omega)$. Таким образом:
$$
\left( \int_m^{M + \varepsilon \mathbf{1}} f(\lambda) \, dE_\lambda f \right)(\omega, t) = f(T_\omega) f(\omega, \cdot)(t).
$$

Теперь покажем ограниченность $f(T)$:
$$
\left\| \int_m^{M + \varepsilon \mathbf{1}} f(\lambda) \, dE_\lambda f \right\|_{2,2}^2 = \int_\Omega \left\| \sum_{n=1}^\infty f(\lambda_n(\omega)) \langle f(\omega, \cdot), x_n(\omega, \cdot) \rangle x_n(\omega, \cdot) \right\|_{L_2(S)}^2 \, d\mu(\omega).
$$
Так как:
$$
\left\| \sum_{n=1}^\infty f(\lambda_n(\omega)) \langle f(\omega, \cdot), x_n(\omega, \cdot) \rangle x_n(\omega, \cdot) \right\|_{L_2(S)}^2 = \sum_{n=1}^\infty |f(\lambda_n(\omega))|^2 |
\langle f(\omega, \cdot), x_n(\omega, \cdot) \rangle|^2
$$
$$
\leq \| f \|_\infty^2 \| f(\omega, \cdot) \|_{L_2(S)}^2,
$$
то:
$$
\left\| \int_m^{M + \varepsilon \mathbf{1}} f(\lambda) \, dE_\lambda f \right\|_{2,2}^2 \leq \| f \|_\infty^2 \int_\Omega \| f(\omega, \cdot) \|_{L_2(S)}^2 \, d\mu(\omega) = \| f \|_\infty^2 \| f \|_{2,2}^2.
$$
Оператор ограничен. Сходимость в СОТ требует, чтобы для $f_n(T) = \sum_{k=1}^n f(\lambda_k(\omega)) \Delta E_{\lambda_k}(\omega)$:
$$
\| (f_n(T) - f(T)) f \|_{2,2}^2 = \int_\Omega \left\| \sum_{k=n+1}^\infty f(\lambda_k(\omega)) \langle f(\omega, \cdot), x_k(\omega, \cdot) \rangle x_k(\omega, \cdot) \right\|_{L_2(S)}^2 \, d\mu(\omega) \to 0.
$$
Так как $f$ непрерывна, а $\lambda_k(\omega) \to 0$, остаток уменьшается, и измеримость $E_\lambda(\omega)$ обеспечивает сходимость интеграла по $\Omega$. Предел в $L_0^h$ следует из структуры модуля Капланского-Гильберта и конечности числа ненулевых $ \lambda_n(\omega) $.

 (4) Для $\lambda \in \operatorname{spm}_{\text{disc}}(T)$, пространство $\operatorname{ker}(\lambda I - T)$ состоит из $f \in L_{2,2}(\Omega \times S)$, таких что:
$$
(\lambda(\omega) I - T_\omega) f(\omega, \cdot) = 0 \quad \text{п.в. } \omega.
$$
На слое, $\operatorname{ker}(\lambda_n(\omega) I - T_\omega)$ порождено $\{ x_{n,k}(\omega, \cdot) \}_{k=1}^{m_n(\omega)}$, где $m_n(\omega) < \infty $. Тогда:
$$
f(\omega, \cdot) = \sum_{k=1}^{m_n(\omega)} c_k(\omega) x_{n,k}(\omega, \cdot), \quad c_k(\omega) \in \mathbb{C}.
$$
В $L_{2,2}(\Omega \times S)$, $f(\omega, t)$ -- измеримое сечение, $c_k(\omega) \in L_0(\Omega)$, так как $x_{n,k}(\omega, \cdot)$ измеримы (Лемма \ref{L1}). Следовательно:
$$
\operatorname{ker}(\lambda I - T) = \left\{ f(\omega, t) = \sum_{k=1}^{m_n(\omega)} c_k(\omega) x_{n,k}(\omega, t) : c_k(\omega) \in L_0(\Omega), \lambda(\omega) = \lambda_n(\omega) \right\}.
$$

Множества $A_n=\left\{\omega: \lambda_n(\omega)=\lambda(\omega)\right\} \in \Sigma$ измеримы, $m_n(\omega)<\infty, \boldsymbol{и} \cup_n A_n$ покрывает $\Omega$ (за исключением множества меры нуль), обеспечивая $\sigma$-конечность. Пространство $\operatorname{ker}(\lambda I-T)$ является модулем над $L_0(\Omega)$, так как для любой $f \in$ $\operatorname{ker}(\lambda I-T)$ и $\alpha \in L_0(\Omega)$ функция $\alpha(\omega) f(\omega, t)$ измерима и принадлежит $\operatorname{ker}(\lambda I-T)$. .

(5) Спектр:
$$
\operatorname{spm}(T) = \operatorname{mix} \left\{ \sum_{n=0}^\infty \chi_{A_n} \tilde{\lambda}_n : \tilde{\lambda}_n(\omega) = \lambda_n(\omega), \tilde{\lambda}_0 = 0, A_n \in \Sigma, A_n \cap A_m = \varnothing \right\}.
$$
Для каждого $\omega\in\Omega$ спектр $\operatorname{sp}(T_\omega) = \{ \lambda_n(\omega) \}_{n=1}^\infty \cup \{0\}$ компактен. Измеримость $\lambda_n(\omega)$ следует из Леммы \ref{L1}. Циклическая компактность $\operatorname{spm}(T)$ доказана в \cite[Следствие 3.7]{kudaybergenov2}. Для $ \lambda \in \operatorname{spm}(T)$, существуют измеримые $A_n \in \Sigma$, такие что:
$$
\lambda(\omega) = \sum_{n=0}^\infty \chi_{A_n}(\omega) \tilde{\lambda}_n(\omega).
$$
Перемешивания покрывают $\operatorname{spm}(T)$, так как $\lambda_n(\omega)$ определяют $\operatorname{sp}(T_\omega)$.
\end{proof}

\section{Пример спектрального разложения частично самосопряженного циклически компактного  частично интегрального оператора }

В данном разделе представлен пример, иллюстрирующий свойства Теоремы \ref{thm:spectral} для самосопряжённого циклически компактного частично интегрального оператора $ T: L_{2,2}(\Omega \times S) \to L_{2,2}(\Omega \times S) $.

Рассмотрим пространства $ \Omega = [0,1]\,\,   S = [0,1] $ с мерами Лебега $ \mu $ и  $ m $ соответственно.
Пусть  $ T: L_{2,2}([0,1]^2) \rightarrow L_{2,2}([0,1]^2) $ частично интегральный оператор
\[
T f(\omega, t) = \int_0^1 k(\omega, t, s) f(\omega, s) \, ds
\]
где:
\[
k(\omega, t, s) = \sum_{n=1}^3 \lambda_n(\omega) \phi_n(t) \phi_n(s)
\]
Здесь,  $ \phi_n(t) = \sqrt{2} \sin(n \pi t) $ — ортонормированный базис в $ L_2([0,1]) $, a  $
    \lambda_1(\omega) = \cos^2\left(\frac{\pi \omega}{2}\right), $  $\lambda_2(\omega) = \frac{1}{2} \sin^2(\pi \omega),$ $ \lambda_3(\omega) = \frac{1}{3} \omega^2$ собственные значения оператора.

$k(\omega, t, s)$   ограничено в $ L_\infty(\Omega \times S^2) $, самосопряжённое $( k(\omega, t, s) = \overline{k(\omega, s, t)} $), и положительно полуопределённое, так как:
\[
\int_0^1 \int_0^1 k(\omega, t, s) x(\omega, s) \overline{x(\omega, t)} \, ds \, dt = \sum_{n=1}^3 \lambda_n(\omega) \left| \int_0^1 x(\omega, t) \phi_n(t) \, dt \right|^2 \geq 0.
\]

Оператор $ T $ действует как:
\[
T f(\omega, t) = \int_0^1 k(\omega, t, s) f(\omega, s) \, ds = \sum_{n=1}^3 \lambda_n(\omega) \langle f(\omega, \cdot), \phi_n \rangle \phi_n(t),
\]
где $ \langle f(\omega, \cdot), \phi_n \rangle = \int_0^1 f(\omega, s) \phi_n(s) \, ds $. Для каждого $ \omega \in\Omega$
\[
T_\omega f_\omega(t) = \sum_{n=1}^3 \lambda_n(\omega) \langle f_\omega, \phi_n \rangle \phi_n(t).
\]
Оператор $T_\omega $ компактен, самосопряжён, с собственными значениями $ \lambda_n(\omega) $ и собственными функциями $ x_n(\omega, t) = \phi_n(t) ,\,\,  n = 1, 2, 3.$

Пусть \[
f(\omega, t) = \omega \sin(\pi t) + \sin(2 \pi t).
\]
Очевидно, что $ f \in L_{2,2}.$

Применим оператор $ T $, интегралы и спектральные конструкции к функции $ f(\omega, t) = \omega \sin(\pi t) + \sin(2 \pi t) $, демонстрируя пункты (1)--(5).

(1)
Вычислим:
\[
\langle f(\omega, \cdot), \phi_n \rangle = \int_0^1 \left( \omega \sin(\pi t) + \sin(2 \pi t) \right) \sqrt{2} \sin(n \pi t) \, dt.
\]
\begin{eqnarray*}
\langle f(\omega, \cdot), \phi_1 \rangle &=& \frac{\omega}{\sqrt{2}},\,\,\, \text{при}\,\,\, n=1,\\
\langle f(\omega, \cdot), \phi_2 \rangle &= & \frac{1}{\sqrt{2}}, \,\,\, \text{при}\,\,\, n=2,\\
\langle f(\omega, \cdot), \phi_3 \rangle &=& 0, \,\,\,\,\,\,\,\,\, \text{при}\,\,\, n=3.
\end{eqnarray*}

Тогда:
\[
T f(\omega, t) = \sum_{n=1}^3 \lambda_n(\omega) \langle f(\omega, \cdot), \phi_n \rangle \phi_n(t) =  \omega \cos^2\left(\frac{\pi \omega}{2}\right) \sin(\pi t) + \frac{1}{2} \sin^2(\pi \omega) \sin(2 \pi t).
\]
Сходимость в $ L_{2,2} $ тривиальна, так как сумма конечна и
\[
\| T f \|_{2,2}^2 = \int_0^1 \left( \frac{1}{2} \omega^2 \cos^4\left(\frac{\pi \omega}{2}\right) + \frac{1}{2} \cdot \frac{1}{4} \sin^4(\pi \omega) \right) \, d\omega < \infty.
\]

(2)
Теорема утверждает, что:
\[
T = \int_m^{M + \varepsilon \mathbf{1}} \lambda \, dE_\lambda,
\]
где интеграл сходится в СОТ.

 Границы спектра имеет вид $m(\omega) = 0,\,\,\,  M(\omega) = \max \{ \cos^2\left(\frac{\pi \omega}{2}\right), \frac{1}{2} \sin^2(\pi \omega), \frac{1}{3} \omega^2 \} \leq 1 $ и
\[
(E_\lambda f)(\omega, t) = \sum_{\{n : \lambda_n(\omega) \leq \lambda(\omega)\}} \langle f(\omega, \cdot), \phi_n \rangle \phi_n(t).
\]
Тогда
\[
\left( \int_0^{M(\omega) + \varepsilon} \lambda \, dE_\lambda f \right)(\omega, t) = \sum_{n=1}^3 \lambda_n(\omega) \langle f(\omega, \cdot), \phi_n \rangle \phi_n(t) = T f(\omega, t).
\]

(3)
Возьмём $ f(\lambda) = \lambda^2 $. Для каждого $\omega\in\Omega$
\[
f(T_\omega) f_\omega = \sum_{n=1}^3 \lambda_n(\omega)^2 \langle f_\omega, \phi_n \rangle \phi_n.
\]
Для $ f(\omega, t) $:
\begin{eqnarray*}
f(T_\omega) f_\omega &= &\cos^4\left(\frac{\pi \omega}{2}\right) \cdot \frac{\omega}{\sqrt{2}} \cdot \sqrt{2} \sin(\pi t) + \left( \frac{1}{2} \sin^2(\pi \omega) \right)^2 \cdot \frac{1}{\sqrt{2}} \cdot \sqrt{2} \sin(2 \pi t)  \\
&=&\omega \cos^4\left(\frac{\pi \omega}{2}\right) \sin(\pi t) + \frac{1}{4} \sin^4(\pi \omega) \sin(2 \pi t).
\end{eqnarray*}
Следовательно
\[
\left( \int_0^{M(\omega) + \varepsilon} \lambda^2 \, dE_\lambda f \right)(\omega, t) = \omega \cos^4\left(\frac{\pi \omega}{2}\right) \sin(\pi t) + \frac{1}{4} \sin^4(\pi \omega) \sin(2 \pi t).
\]

(4)
Для $\lambda(\omega)=\lambda_1(\omega)=\cos ^2\left(\frac{\pi \omega}{2}\right)$
\[
\operatorname{ker}\left(\lambda_1(\omega) I-T_\omega\right)=\operatorname{span}\{\sqrt{2} \sin (\pi t)\},
\]
 с кратностью 1.

  В $L_{2,2}(\Omega \times S) $ имеем $\operatorname{ker}(\lambda I-T)=\left\{c_1(\omega) \sqrt{2} \sin (\pi t): c_1(\omega) \in L_0([0,1])\right\}$. Функция $f(\omega, t)=$ $\omega \sin (\pi t)+\sin (2 \pi t)$ не принадлежит $\operatorname{ker}(\lambda I-T)$ из-за члена $\sin (2 \pi t)$. Аналогично, для $\lambda(\omega)=\lambda_2(\omega)=\frac{1}{2} \sin ^2(\pi \omega),$
   \[\operatorname{ker}\left(\lambda_2(\omega) I-T_\omega\right)=\operatorname{span}\{\sqrt{2} \sin (2 \pi t)\}.\] Пространство $\operatorname{ker}(\lambda I-T)$ -- модуль над $L_0(\Omega)$, так как $\alpha(\omega) f(\omega, t) \in \operatorname{ker}(\lambda I-T)$ для $f \in \operatorname{ker}(\lambda I-$ $T),\,\, \alpha \in L_0(\Omega) . $

(5) Для каждого $\omega\in\Omega$ спектр оператора $ T_\omega$ имеет следующий вид:
  \[
  \operatorname{sp}(T_\omega) = \left\{ \cos^2\left(\frac{\pi \omega}{2}\right), \frac{1}{2} \sin^2(\pi \omega), \frac{1}{3} \omega^2, 0 \right\} ,
  \]
  а циклический модулярный спектр оператора $T$
\[
\operatorname{spm}(T) = \left\{ \lambda \in L_0^h : \lambda(\omega) \in \left\{ \cos^2\left(\frac{\pi \omega}{2}\right), \frac{1}{2} \sin^2(\pi \omega), \frac{1}{3} \omega^2, 0 \right\} \text{ для \,п.в. } \omega \right\}.
\]
Для $ A_1 = [0, \frac{1}{3}), A_2 = [\frac{1}{3}, \frac{2}{3}), A_3 = [\frac{2}{3}, 1], A_0 = \varnothing $:
\[
\lambda(\omega) = \chi_{A_1}(\omega) \cos^2\left(\frac{\pi \omega}{2}\right) + \chi_{A_2}(\omega) \frac{1}{2} \sin^2(\pi \omega) + \chi_{A_3}(\omega) \frac{1}{3} \omega^2
\]
принадлежит $ \operatorname{spm}(T) $.  Согласно \cite[Следствие~3.7]{kudaybergenov2}  $ \operatorname{spm}(T) $ циклически компактен.

\vspace{1cm}

\textsl{Каримберген Кадирбергенович Кудайбергенов$^1$},

\textsl{Аллабай Джалгасович Арзиев$^2$},

\textsl{Парахатдиин Рахман улы Орынбаев$^3$}.

\vspace{0.3cm}
$^{1,3}$ Владикавказский научный центр РАН,

ул. Вильямся 1, с. Михайловское, г.Владикавказ, 363110, Россия. E-mail: karim2006@mail.ru.

\vspace{0.3cm}
$^2$  Институт математики имени В.И.Романовского АН РУз,

ул. Университетская 9, Алмазарский район, г. Ташкент, 100174, Узбекистан.

E-mail: allabayarziev@inbox.ru, \,paraxatorinbaev@gmail.com.

\vspace{0.3cm}
$^2$ Каракалпакский государственный университет имени Бердаха,

ул. Ч.Абдирова 1, г. Нукус, 230100, Узбекистан. E-mail: allabayarziev@karsu.uz.

\label{lastpage}

\end{document}